\documentclass[11pt]{amsart}
\PassOptionsToPackage{hyphens}{url}
\usepackage{amsmath,amssymb,amsthm,mathtools,url}
\usepackage[margin=1.1in]{geometry}
\usepackage{booktabs,needspace}
\usepackage[section]{placeins}
\usepackage{hyperref}
\hypersetup{hidelinks,
  pdftitle={An inclusion coupling of prefix laws and the Chan--Pak rank inequality for all finite posets},
  pdfauthor={Jukka Virtanen}}

\newtheorem{theorem}{Theorem}[section]
\newtheorem{lemma}[theorem]{Lemma}
\newtheorem{proposition}[theorem]{Proposition}
\newtheorem{corollary}[theorem]{Corollary}
\newtheorem{conjecture}[theorem]{Conjecture}
\newtheorem*{fishburn}{Fishburn's inequality}
\theoremstyle{definition}
\newtheorem{definition}[theorem]{Definition}
\theoremstyle{remark}
\newtheorem{remark}[theorem]{Remark}

\newcommand{\LE}{\operatorname{LE}}
\newcommand{\pos}{\operatorname{pos}}

\newcommand{\inc}{\parallel}
\newcommand{\Prob}{\mathbb{P}}
\newcommand{\E}{\mathbb{E}}
\newcommand{\Z}{\mathbb{Z}}
\newcommand{\Id}{\mathrm{Id}}
\newcommand{\last}{\operatorname{last}}
\newcommand{\first}{\operatorname{first}}

\title[The Chan--Pak rank inequality]{An inclusion coupling of prefix laws and the Chan--Pak rank inequality for all finite posets}

\author{Jukka Virtanen}
\address{Independent researcher, Irvine, California, USA}
\email{semisimplemath@gmail.com}
\thanks{Ph.D.\ in Mathematics, University of California, Los Angeles. The author was previously a postdoctoral researcher and, for eight years, an adjunct faculty member at UCLA}


\date{October 17, 2026}

\begin{document}
\setlength{\emergencystretch}{3em}

\begin{abstract}
Let $P$ be a finite poset, let $x\in P$, and let $N_k$ count its linear extensions in which $x$ occupies position $k$. We prove Chan and Pak's conjecture that, for all integers $a,h\ge1$,
\[
(a+h-1)N_{a+h-1}N_a\ \ge\ (a-1)N_{a-1}N_{a+h}+hN_aN_{a+h}.
\]
The key result is an inclusion coupling of the laws of the set of elements preceding $x$, conditional on consecutive positions of $x$. We construct reversible exchange chains on the corresponding layers of order ideals and couple their generators so that the two prefix sets remain nested. Fishburn's correlation inequality gives monotonicity of endpoint probabilities; their unit total masses provide the rate bounds needed for the coupling. A stationary distribution of the coupled chain then has the required marginals. Normalized cover observables yield a four-term nonnegative certificate for the case $h=1$, and the general case follows by telescoping. The argument applies at every width and gives a proof of Stanley's log-concavity inequality with Fishburn's inequality as its sole external inequality.
\end{abstract}

\maketitle

\section{Introduction}\label{sec:intro}

\subsection{The conjecture}
Let $P$ be a finite poset on $n$ elements and fix $x\in P$. For $k\in\Z$ let
\[
N_k=\bigl|\{L\in\LE(P):\pos_L(x)=k\}\bigr|,
\]
with $N_k=0$ outside $1,\dots,n$. Stanley's inequality~\cite{Stanley1981} states that $N_k^2\ge N_{k-1}N_{k+1}$. In their survey~\cite[\S10.2]{ChanPakSurvey}, Chan and Pak state the following ``surprising inequality'' as a conjecture.

\begin{conjecture}[{Chan--Pak \cite[Conjecture~10.4]{ChanPakSurvey}}]\label{conj:CP104}
Let $P=(X,\prec)$ be a poset on $|X|=n$ elements and fix $z\in X$. Then for all integers $a,i\ge1$,
\[
\begin{aligned}(a+i-1)\,N(P,z,a+i-1)\,N(P,z,a)\ \ge\ {}&(a-1)\,N(P,z,a-1)\,N(P,z,a+i)\\&+i\,N(P,z,a+i)\,N(P,z,a).\end{aligned}
\]
\end{conjecture}

With $z=x$ and $i=h$, this is the inequality
\[
(a+h-1)N_{a+h-1}N_a\ \ge\ (a-1)N_{a-1}N_{a+h}+hN_aN_{a+h}\qquad(a,h\ge1),\tag{T}
\]
and its case $h=1$ is
\[
kN_k^2\ \ge\ (k-1)N_{k-1}N_{k+1}+N_kN_{k+1}\qquad(k\ge1).\tag{C}
\]
We call \textup{(C)} \emph{property~C}.

The survey records two facts about Conjecture~\ref{conj:CP104}, both attributed to Chan--Pak--Panova \cite[\S15.1]{ChanPakSurvey}: the Cross-Product Conjecture implies Conjecture~\ref{conj:CP104} \cite[Theorem~10.5]{ChanPakSurvey}; and Conjecture~\ref{conj:CP104} implies Stanley's inequality together with the equivalence (1)$\Leftrightarrow$(2) in the equality conditions of \cite[Theorem~10.2]{ChanPakSurvey}; see \cite[Proposition~10.6]{ChanPakSurvey}. The survey also reports that computer searches for a counterexample to Conjecture~\ref{conj:CP104} have not succeeded \cite[\S16.3]{ChanPakSurvey}.

\subsection{Main result}
Throughout, $P=(V,<)$ is a finite strict poset, $x\in V$, $n=|V|$, and
\[
D=\{y\in V:y<x\},\qquad R=\{y\in V:y\inc x\},
\]
where $y\inc x$ means that $y\ne x$ and $y,x$ are incomparable; $R$ carries the induced order. For $S\subseteq V$ we write $e(S)$ for the number of linear extensions of the induced subposet $P[S]$, with $e(\varnothing)=1$. For an order ideal $J$ of $R$ we put $I=D\cup J$ and $S_I=V\setminus(I\cup\{x\})$. Such sets $I$ are called \emph{feasible prefixes}; they are exactly the sets of elements preceding $x$ in linear extensions of $P$ (Theorem~\ref{thm:endpoint}).

For a feasible prefix size $p$ let $\mu_p$ be the law of the set of elements preceding $x$ under the uniform measure on $\LE(P)$ conditioned on $\pos(x)=p+1$. Let $L_k$ be the number of linear extensions with $x$ at position $k$ whose immediate predecessor is a lower cover of $x$, and $U_k$ the number whose immediate successor is an upper cover of $x$. Set $L_k=U_k=0$ outside $1,\dots,n$; also $L_1=0$ and $U_n=0$, since the corresponding neighbour does not exist. Put $\ell_k=L_k/N_k$ and $u_k=U_k/N_k$ when $N_k>0$.

\begin{theorem}\label{thm:main}
Let $P$ be a finite poset and $x\in P$.
\begin{enumerate}
\item \textup{(Prefix inclusion coupling.)} For every two consecutive feasible prefix sizes $p,p+1$, there is a coupling of $\mu_p$ and $\mu_{p+1}$ supported on pairs $I\subset J$ with $|J\setminus I|=1$.
\item \textup{(Cover frequencies.)} For consecutive ranks $k,k+1$ with $N_k,N_{k+1}>0$,
\[
(k-1)\ell_k\ge k\ell_{k+1},\qquad(n-k)u_k\le(n-k-1)u_{k+1},
\]
and $N_k-U_k=N_{k+1}-L_{k+1}>0$. Consequently $N_{k+1}/N_k=(1-u_k)/(1-\ell_{k+1})$, and the ratios $N_{k+1}/N_k$ are nonincreasing in $k$ on the support.
\item \textup{(Property C with a defect bound.)} For every $k\ge1$,
\[
D_k:=kN_k^2-(k-1)N_{k-1}N_{k+1}-N_kN_{k+1}\ \ge\ N_kU_k\ \ge\ 0.
\]
\item \textup{(Inequality (T).)} For all integers $a,h\ge1$,
\[
(a+h-1)N_{a+h-1}N_a\ \ge\ (a-1)N_{a-1}N_{a+h}+hN_aN_{a+h}.
\]
\end{enumerate}
No hypothesis on the width of $P$ is needed.
\end{theorem}

\begin{corollary}\label{cor:CP104}
Conjecture~\ref{conj:CP104} holds for every finite poset $P$ and every element $z\in P$.
\end{corollary}

The coupling in (1) comes from an explicit Markov chain. For each feasible size $p$, Section~\ref{sec:exchange} defines a reversible \emph{exchange chain} on the size-$p$ feasible prefixes which exchanges $a\in I\cap R$ with $b\in S_I\cap R$ when $a\inc b$, at rate $\last_{P[I]}(a)\first_{P[S_I]}(b)$ (Definition~\ref{def:endpoint}); its unique invariant law is $\mu_p$. Section~\ref{sec:coupled} couples the exchange chains at sizes $p$ and $p+1$ so that inclusion is preserved, and Section~\ref{sec:inclusion} projects a stationary law of the coupled chain to obtain (1).

\subsection{Consequences}
\emph{Stanley's inequality.} By \cite[Proposition~10.6]{ChanPakSurvey}, property~C for $P$ and for its order dual implies Stanley's inequality $N_k^2\ge N_{k-1}N_{k+1}$: if $N_k<N_{k+1}$ one uses \textup{(C)} for $P$, if $N_k<N_{k-1}$ one uses \textup{(C)} for the order dual, and otherwise $N_k$ is a local maximum and the inequality is immediate. (Log-concavity also follows directly from part~(2) of Theorem~\ref{thm:main}.) The proof of Theorem~\ref{thm:main} uses a single external inequality, the factorial form of Fishburn's correlation inequality~\cite{Fishburn1984} (Section~\ref{sec:prelim}). Fishburn's proof uses the Ahlswede--Daykin four functions theorem~\cite{AD78}; the inequality also follows from a theorem of Shepp~\cite[Theorem~2]{Shepp80}, as explained in~\cite[Lemma~10]{Brightwell88} (see \cite[\S7.2]{ChanPakSurvey}). Everything else is the prefix--$x$--suffix factorization of linear extensions, a finite reversible Markov chain on one layer of order ideals, a coupled generator on pairs of nested ideals, the existence of a stationary law for a finite stochastic matrix, and counting by adjacent swaps. Theorem~\ref{thm:main} therefore gives a proof of Stanley's inequality that does not use the Alexandrov--Fenchel inequality, as in Stanley's original proof~\cite{Stanley1981}, or the combinatorial atlas.

\emph{Further conclusions.} Besides \textup{(C)} and \textup{(T)}, the proof gives the coupling itself: for consecutive positions of $x$, the laws of the set of elements preceding $x$ are stochastically ordered with respect to inclusion, in the strong form of a coupling supported on one-element inclusions. Gluing the adjacent couplings gives a simultaneous nested coupling over all feasible ranks (Corollary~\ref{cor:all-ranks}). The proof also gives the monotonicity of the normalized lower- and upper-cover adjacency frequencies, an explicit four-term nonnegative expression for the defect $D_k$ (equation~\eqref{eq:certificate}), and the bound $D_k\ge N_kU_k$.

\subsection{Relation to previous work}
\emph{The conjecture.} We are not aware of a previous proof of Conjecture~\ref{conj:CP104} beyond special cases. The reduction of \cite[Theorem~10.5]{ChanPakSurvey} requires the Cross-Product Conjecture for an auxiliary poset whose width exceeds that of $P$ by one; special cases of the Cross-Product Conjecture are known, including width two~\cite{CPP22CPC}, the case $a=b=1$~\cite{BFT95}, and a vanishing case~\cite{CPP25}.

\emph{Markov chains on order ideals.} Leake, Lindberg and Oveis Gharan~\cite[\S6.2]{LLOG25} view linear extensions as maximal chains in the distributive lattice of order ideals and re-prove Stanley's inequality by local spectral expansion (trickle-down). Their spectral argument treats Stanley's inequality and related first-hit distributions. In contrast, the key step here is an explicit inclusion-preserving coupling of the marked-rank prefix laws.

\emph{Injective proofs.} Chan and Pak~\cite[\S10.9]{ChanPak2024AF} note that Stanley's inequality most likely cannot be proved by a direct injection. The argument here is not an injection: the coupling is a stationary law of a finite chain, and the certificate for the defect is an identity among rational quantities. The rates involve exact extension counts, and the proof establishes neither an efficient sampling algorithm nor a polynomial-time injection. We draw no complexity-theoretic conclusion.

\subsection{Computer checks}
Independently of the proof, we checked inequality \textup{(T)} in exact arithmetic for every element of every poset with at most $10$ elements, and of every width-three poset with at most $12$ elements. We checked the existence of the inclusion coupling of Theorem~\ref{thm:main}(1) by an exact maximum-flow computation for every poset with at most $8$ elements. For the same posets we also checked every step of the construction in exact rational arithmetic: the exchange chains, the coupled generator, and the marginals of its stationary laws. No failures occurred. Details and self-contained scripts reproducing these computations are given in Appendix~\ref{sec:census}.

\subsection{Organization}
Section~\ref{sec:prelim} fixes notation and states Fishburn's inequality. Section~\ref{sec:endpoint} proves the prefix factorization and the monotonicity of endpoint probabilities under a one-element enlargement. Sections~\ref{sec:exchange}--\ref{sec:inclusion} construct the exchange chains, the coupled generator and the prefix inclusion coupling. Section~\ref{sec:deduction} deduces parts (2)--(4) of Theorem~\ref{thm:main}. Appendix~\ref{sec:census} gives details of the computer checks.

\section{Notation and Fishburn's inequality}\label{sec:prelim}

A finite strict poset $(P,<)$ is a finite set with an irreflexive transitive relation. A \emph{linear extension} of $P$ is a sequence $L=(x_1,\dots,x_n)$ listing every element of $P$ exactly once such that $u<v$ implies that $u$ occurs before $v$; $\pos_L(p)$ is the (one-based) index of $p$ in $L$. We write $\LE(P)$ for the set of linear extensions and $e(P)=|\LE(P)|$. All subposets are induced: for $S\subseteq P$, $P[S]$ retains every comparison between elements of $S$, $e(S)=e(P[S])$, and $e(\varnothing)=1$. We write $P-S$ for $P[P\setminus S]$ and $P-v$ for $P-\{v\}$. For sets, $I-a+b$ means $(I\setminus\{a\})\cup\{b\}$, and $I-a=I\setminus\{a\}$. For a real number $t$, write $t_+=\max\{t,0\}$.

An \emph{order ideal} is a subset $I$ such that $z<y\in I$ implies $z\in I$; a \emph{filter} is the complement of an order ideal. We write $a\lessdot b$ ($a$ is \emph{covered} by $b$) if $a<b$ and no $c$ satisfies $a<c<b$; a \emph{lower cover} of $x$ is an element $v\lessdot x$, and an \emph{upper cover} is an element $v$ with $x\lessdot v$. The order dual $P^*$ reverses every comparison and exchanges ideals with filters. Reversing words gives a bijection $\LE(P)\to\LE(P^*)$, so duality preserves extension counts.

\begin{fishburn}[{\cite[Lemma, p.~130]{Fishburn1984}; see also \cite[Theorem~3.1]{ChanPak2023} and \cite[Theorem~7.4]{ChanPakSurvey}}]
For order ideals $A,B$ of a finite poset,
\[
e(A\cup B)\,e(A\cap B)\,|A|!\,|B|!\ \ge\ e(A)\,e(B)\,|A\cup B|!\,|A\cap B|!.
\]
\end{fishburn}
By order duality the same inequality holds for two filters $A,B$.

\Needspace{10\baselineskip}
\section{Endpoint probabilities}\label{sec:endpoint}

\begin{definition}\label{def:endpoint}
Let $H$ be a finite strict poset. For $a\in H$ put $\last_H(a)=e(H-a)/e(H)$ if $a$ is maximal in $H$ and $\last_H(a)=0$ otherwise; for $b\in H$ put $\first_H(b)=e(H-b)/e(H)$ if $b$ is minimal in $H$ and $\first_H(b)=0$ otherwise. For $H=\varnothing$ all endpoint probabilities are zero, and endpoint probabilities are extended by zero outside $H$.
\end{definition}
For nonempty $H$, $\last_H$ is the law of the last element of a uniformly random linear extension of $H$, and $\first_H$ the law of its first element; each sums to one over $H$.

\begin{theorem}\label{thm:endpoint}
Let $J$ be an order ideal of $R$, $I=D\cup J$ and $S=V\setminus(I\cup\{x\})$.
\begin{enumerate}
\item $I$ and $I\cup\{x\}$ are order ideals of $P$. Concatenation $(\pi,\tau)\mapsto\pi\,x\,\tau$ is a bijection from $\LE(P[I])\times\LE(P[S])$ onto the linear extensions of $P$ in which the set of elements preceding $x$ is $I$. In particular this set has $w(I)=e(I)\,e(S)$ elements.
\item The set of elements preceding $x$ in any linear extension is a feasible prefix. The feasible prefix sizes are exactly $|D|,\dots,|D|+|R|$, $N_{p+1}=\sum_{|I|=p}w(I)>0$ for each of them, and $\mu_p(I)=w(I)/N_{p+1}$.
\item Let $I\subset J'=I\cup\{d\}$ be feasible prefixes with $d\in R$; put $p=|I|$ and $q=|S|$. Then $d$ is maximal in $J'$ and minimal in $S$. Write
\begin{align*}
\lambda_a&=\last_{P[I]}(a),& \lambda'_a&=\last_{P[J']}(a)&&(a\in J'),\\
\beta_b&=\first_{P[S]}(b),& \beta'_b&=\first_{P[S-d]}(b)&&(b\in S).
\end{align*}
The endpoint probabilities satisfy
\[
(p+1)\lambda'_a\le p\lambda_a\quad(a\in I),\qquad
(q-1)\beta'_b\ge q\beta_b\quad(b\in S-d).
\]
In particular, $\lambda'_a\le\lambda_a$ and $\beta'_b\ge\beta_b$ on these index sets. When $p=0$ or $q=1$, respectively, the relevant index set is empty. Moreover,
\[
\sum_{a\in I}(\lambda_a-\lambda'_a)=\lambda'_d\quad(p>0),\qquad
\sum_{b\in S-d}(\beta'_b-\beta_b)=\beta_d\quad(q>1).
\]
All summands are nonnegative. These sums include every endpoint label, not only the elements of $R$.
\end{enumerate}
\end{theorem}

\begin{proof}
(1) If $y\in D$, every predecessor of $y$ is below $x$, hence in $D$. If $y\in J$ and $z<y$, then $z\ne x$ (else $x<y$) and $z\not>x$ (else $x<y$), so $z\in D$ or $z\in R$, and in the latter case $z\in J$ because $J$ is an ideal of $R$. Thus $I$ is an ideal, and so is $I\cup\{x\}$, since the predecessors of $x$ form $D\subseteq I$. A linear extension with prefix set $I$ before $x$ has the unique form $\pi\,x\,\tau$ with $\pi\in\LE(P[I])$ and $\tau\in\LE(P[S])$. Conversely, every such concatenation is a linear extension: since $I\cup\{x\}$ is an ideal, no element of $S$ lies below an element of $I\cup\{x\}$, and all relations inside the blocks are respected.

(2) Let $I$ be the set of elements preceding $x$ in some linear extension. It is an ideal containing $D$ and avoiding $x$, and it avoids every $y>x$; hence $I=D\cup J$ with $J=I\cap R$, and $J$ is an ideal of $R$. Summing the fibre sizes of (1) gives $N_{p+1}$. Ideals of $R$ have sizes $0,\dots,|R|$, and every size occurs (take the initial segments of a linear extension of $R$), so the feasible sizes are as stated and $N_{p+1}>0$ for each of them. The conditional law is uniform on the $N_{p+1}$ extensions, so $\mu_p(I)=w(I)/N_{p+1}$.

(3) Every predecessor of $d$ lies in the ideal $J'$ and differs from $d$, so it lies in $I$. Hence $d$ is minimal in $S$. If $d<y$ with $y\in J'$, then $y\ne d$, so $y\in I$; also $y\in R$, since $y\in D$ would give $d<x$. As $I\cap R$ is an ideal of $R$, this forces $d\in I$, a contradiction. So $d$ is maximal in $J'$.

Fix $a\in I$. If $a$ is not maximal in $J'$, then $\lambda'_a=0$. Otherwise $a$ is maximal in $I$, and $I$ and $J'-a$ are ideals of $P$ of size $p\ge1$ with union $J'$ and intersection $I-a$. Fishburn's inequality gives
\[
e(J')\,e(I-a)\,(p!)^2\ \ge\ e(I)\,e(J'-a)\,(p+1)!\,(p-1)!,
\]
that is, $\lambda'_a=e(J'-a)/e(J')\le\frac{p}{p+1}\,e(I-a)/e(I)=\frac{p}{p+1}\lambda_a$.

Fix $b\in S-d$ (so $q\ge2$). If $b$ is not minimal in $S$, then $\beta_b=0$. Otherwise $b$ is minimal in $S-d$, and $S-b$ and $S-d$ are filters of $P[S]$ of size $q-1$ with union $S$ and intersection $S-\{b,d\}$. Fishburn's inequality for filters gives
\[
e(S)\,e(S-\{b,d\})\,\bigl((q-1)!\bigr)^2\ \ge\ e(S-b)\,e(S-d)\,q!\,(q-2)!,
\]
that is, $\beta'_b=e(S-\{b,d\})/e(S-d)\ge\frac{q}{q-1}\,e(S-b)/e(S)=\frac{q}{q-1}\beta_b$.

If $p>0$, then $\sum_{a\in I}\lambda_a=1$ and $\sum_{a\in I}\lambda'_a=1-\lambda'_d$, which gives the first identity. If $q>1$, then $\sum_{b\in S-d}\beta'_b=1$ and $\sum_{b\in S-d}\beta_b=1-\beta_d$, which gives the second. The summands are nonnegative by the pointwise inequalities.
\end{proof}

\section{The exchange chain on a prefix layer}\label{sec:exchange}

For a feasible prefix size $p$ let $\mathcal F_p$ be the set of feasible prefixes of size $p$. For $I\in\mathcal F_p$, $a\in I\cap R$ and $b\in S_I\cap R$ with $a\inc b$, define the rate
\begin{equation}\label{eq:rate}
q_p\bigl(I,\,I-a+b\bigr)=\last_{P[I]}(a)\,\first_{P[S_I]}(b),
\end{equation}
set all other off-diagonal rates to zero, and let the diagonal entries be minus the row sums. Positive rates lead to states in $\mathcal F_p$, as checked below. Thus these rates define a generator $Q_p$ directly on prefix sets; no Markov projection of a chain on linear extensions is assumed.

\begin{proposition}\label{prop:exchange}
For each feasible $p$, $Q_p$ is a generator on $\mathcal F_p$ with total exit rates at most one. It is reversible with respect to $\mu_p$, its transition graph is connected, and $\mu_p$ is its unique invariant probability.
\end{proposition}

\begin{proof}
\emph{Legality.} If $q_p(I,I-a+b)>0$, then $a$ is maximal in $I$, $b$ is minimal in $S_I$ and $a\inc b$, so $K=I-a+b$ is again an ideal containing $D$ and avoiding $x$ and every element above $x$; thus $K\in\mathcal F_p$. Conversely, let $I,K\in\mathcal F_p$ with $K=I-a+b$. Then $a,b\in R$, since both prefixes contain $D$ and every feasible prefix lies in $D\cup R$. A successor of $a$ in $I$ would lie in $I-a\subseteq K$ and force $a\in K$, so $a$ is maximal in $I$. Every predecessor of $b$ lies in the ideal $K$ and differs from $a$ and $b$, so it lies in $I$; hence $b$ is minimal in $S_I$. Finally, $a<b$ would force $a\in K$, and $b<a$ would force $b\in I$; so $a\inc b$. By symmetry $b$ is maximal in $K$ and $a$ is minimal in $S_K$. Hence $q_p(I,K)>0$ if and only if $q_p(K,I)>0$. The total exit rate from $I$ is at most $\sum_{a,b}\last_{P[I]}(a)\first_{P[S_I]}(b)\le1$.

\emph{Detailed balance.} With $K=I-a+b$ we have $S_K=S_I-b+a$, and
\[
\begin{aligned}
w(I)\,q_p(I,K)&=e(I)e(S_I)\cdot\frac{e(I-a)}{e(I)}\cdot\frac{e(S_I-b)}{e(S_I)}=e(I-a)\,e(S_I-b)\\
&=e(K-b)\,e(S_K-a)=w(K)\,q_p(K,I).
\end{aligned}
\]
Hence $\mu_p(I)q_p(I,K)=\mu_p(K)q_p(K,I)$ for all $I\ne K$, and $\mu_p$ is invariant.

\emph{Connectivity.} Let $I\ne K$ in $\mathcal F_p$. Choose $a$ maximal in $I\setminus K$ and $b$ minimal in $K\setminus I$. A successor of $a$ in $I$ would lie in $I\setminus K$ (as $K$ is an ideal not containing $a$), contradicting the choice of $a$; so $a$ is maximal in $I$. Every predecessor of $b$ lies in $K$, and not in $K\setminus I$ by minimality, so it lies in $I$; hence $b$ is minimal in $S_I$. Finally $a<b$ would put $a$ in $K$, and $b<a$ would put $b$ in $I$; so $a\inc b$. Thus $I\to I-a+b$ is a positive-rate move, and $I-a+b$ has smaller symmetric difference with $K$. Iterating connects $I$ to $K$.

\emph{Uniqueness.} Let $\nu$ be an invariant probability, $\nu Q_p=0$, and put $h=\nu/\mu_p$. For every $K\in\mathcal F_p$, by detailed balance (which holds trivially on the diagonal),
\[
\begin{aligned}
0&=\sum_I\nu(I)q_p(I,K)
 =\sum_Ih(I)\mu_p(I)q_p(I,K)\\
 &=\mu_p(K)\sum_Iq_p(K,I)h(I)
 =\mu_p(K)\,(Q_ph)(K).
\end{aligned}
\]
Since $\mu_p(K)>0$, we get $Q_ph=0$. At a state $K$ where $h$ is maximal, $0=(Q_ph)(K)=\sum_{I\ne K}q_p(K,I)\bigl(h(I)-h(K)\bigr)$ is a sum of nonpositive terms, so $h(I)=h(K)$ for every neighbour $I$ of $K$. By connectivity $h$ is constant, and normalization gives $\nu=\mu_p$.
\end{proof}

\section{The coupled generator on nested pairs}\label{sec:coupled}

Let $p,p+1$ be consecutive feasible prefix sizes and let
\[
\mathcal C_p=\{(I,J):I\in\mathcal F_p,\ J\in\mathcal F_{p+1},\ I\subset J\}.
\]
Fix $(I,J)\in\mathcal C_p$, write $J=I\cup\{d\}$ with $d\in R$, and put $S=S_I$, so that $S_J=S-d$. Use $\lambda,\lambda',\beta,\beta'$ from Theorem~\ref{thm:endpoint}(3), with $J'=J$. Throughout this section, $a$ ranges over $I\cap R$ and $b$ over $(S-d)\cap R$. The nonspecial exchanges $a\to b$ have lower and upper rates
\[
r_{ab}=\mathbf 1_{a\inc b}\lambda_a\beta_b,
\qquad t_{ab}=\mathbf 1_{a\inc b}\lambda'_a\beta'_b.
\]
Define their excesses and the special rates by
\begin{align*}
E_{ab}&=(r_{ab}-t_{ab})_+,& F_{ab}&=(t_{ab}-r_{ab})_+,\\
\rho_a&=\mathbf 1_{a\inc d}\lambda_a\beta_d,&
\sigma_b&=\mathbf 1_{d\inc b}\lambda'_d\beta'_b.
\end{align*}
Here $\rho_a$ is the lower rate for $a\to d$, and $\sigma_b$ is the upper rate for $d\to b$.

\begin{lemma}[Rate budgets]\label{lem:budget}
For every $b$, $\sum_aE_{ab}\le\sigma_b$; for every $a$, $\sum_bF_{ab}\le\rho_a$.
\end{lemma}

\begin{proof}
The endpoint inequalities give
\[
E_{ab}\le\mathbf 1_{a\inc b}(\lambda_a-\lambda'_a)\beta_b,
\qquad
F_{ab}\le\mathbf 1_{a\inc b}\lambda'_a(\beta'_b-\beta_b).
\]
Fix $b$. If $E_{ab}>0$ for some $a$, then $p>0$, $q=|S|>1$ and $\beta_b>0$. Thus $b$ and $d$ are distinct minimal elements of $S$, so $d\inc b$. The first endpoint-mass identity gives
\begin{align*}
\sum_{a\in I\cap R}E_{ab}
&\le \beta_b\sum_{a\in I\cap R}(\lambda_a-\lambda'_a)\\
&\le \beta_b\sum_{a\in I}(\lambda_a-\lambda'_a)
 =\beta_b\lambda'_d
 \le\beta'_b\lambda'_d=\sigma_b.
\end{align*}
The second inequality uses nonnegativity of every endpoint difference, including those indexed by $D$. If no $E_{ab}$ is positive, the bound is immediate.

Similarly, fix $a$ with $F_{ab}>0$ for some $b$. Then $p>0$, $q>1$ and $\lambda'_a>0$. The elements $a,d$ are distinct maxima of $J$, so $a\inc d$. The second endpoint-mass identity gives
\begin{align*}
\sum_{b\in(S-d)\cap R}F_{ab}
&\le\lambda'_a\sum_{b\in(S-d)\cap R}(\beta'_b-\beta_b)\\
&\le\lambda'_a\sum_{b\in S-d}(\beta'_b-\beta_b)
 =\lambda'_a\beta_d
 \le\lambda_a\beta_d=\rho_a.
\end{align*}
Again, a zero excess family causes no difficulty. Thus the available \emph{legal} special moves have enough rate to absorb both kinds of excess.
\end{proof}

Define the off-diagonal rates of $A$ out of $(I,J)$ by the table below. Use rows (i)--(iii) once for each pair $(a,b)$, row (iv) once for each $a$, and row (v) once for each $b$. Omit zero rates, add rates when destinations coincide, and set the diagonal entry to minus the row sum:
\begin{center}
\begin{tabular}{lll}
\toprule
& destination & rate\\
\midrule
(i) & $(I-a+b,\ J-a+b)$ & $\min(r_{ab},t_{ab})$\\
(ii) & $(I-a+b,\ J-d+b)$ & $E_{ab}$\\
(iii) & $(I-a+d,\ J-a+b)$ & $F_{ab}$\\
(iv) & $(I-a+d,\ J)$ & $\rho_a-\sum_bF_{ab}$\\
(v) & $(I,\ J-d+b)$ & $\sigma_b-\sum_aE_{ab}$\\
\bottomrule
\end{tabular}
\end{center}

\begin{proposition}\label{prop:coupled}
$A$ is a generator on $\mathcal C_p$, with total exit rates at most two. For every function $f$ on $\mathcal F_p$ and $g$ on $\mathcal F_{p+1}$,
\begin{equation}\label{eq:marginal}
A(f\circ\mathrm{pr}_1)=(Q_pf)\circ\mathrm{pr}_1,\qquad A(g\circ\mathrm{pr}_2)=(Q_{p+1}g)\circ\mathrm{pr}_2,
\end{equation}
where $\mathrm{pr}_1(I,J)=I$ and $\mathrm{pr}_2(I,J)=J$.
\end{proposition}

\begin{proof}
All rates are nonnegative by Lemma~\ref{lem:budget}.

\emph{Destinations lie in $\mathcal C_p$.} Consider a row with positive assigned rate. Its coordinate moves are positive-rate moves of $Q_p$ or $Q_{p+1}$: in (i) both $r_{ab}$ and $t_{ab}$ are positive; in (ii) $r_{ab}>0$, and the upper move $d\to b$ has rate $\sigma_b>0$, because $\sigma_b\ge\sum_aE_{ab}>0$; in (iii) $t_{ab}>0$, and the lower move $a\to d$ has rate $\rho_a\ge\sum_bF_{ab}>0$; in (iv) and (v) the moves have rates $\rho_a>0$ and $\sigma_b>0$. By Proposition~\ref{prop:exchange} the coordinates of every destination are feasible prefixes. Inclusion is preserved: in (i) $J-a+b=(I-a+b)\cup\{d\}$; in (ii) $J-d+b=I\cup\{b\}\supset I-a+b$; in (iii) $J-a+b=(I-a+d)\cup\{b\}$; in (iv) $J=(I-a+d)\cup\{a\}$; in (v) $J-d+b=I\cup\{b\}$.

\emph{Marginals.} Group the joint moves by their effect on the first coordinate. The lower move $a\to b$ ($b\ne d$) occurs in (i) and (ii), with total rate $\min(r_{ab},t_{ab})+E_{ab}=r_{ab}$. The lower move $a\to d$ occurs in (iii) and (iv), with total rate $\sum_bF_{ab}+\rho_a-\sum_bF_{ab}=\rho_a$. Row (v) leaves $I$ fixed. Since distinct lower moves lead to distinct sets, for $f$ on $\mathcal F_p$,
\[
\begin{aligned}
A(f\circ\mathrm{pr}_1)(I,J)
 &=\sum_{(I',J')}A\bigl((I,J),(I',J')\bigr)\bigl(f(I')-f(I)\bigr)\\
 &=\sum_{I'}q_p(I,I')\bigl(f(I')-f(I)\bigr)
 =(Q_pf)(I),
\end{aligned}
\]
where joint moves that fix $I$ contribute zero. This includes the diagonal, since both generators have zero row sums. Similarly, the upper move $a\to b$ occurs in (i) and (iii), with total rate $\min(r_{ab},t_{ab})+F_{ab}=t_{ab}$; the upper move $d\to b$ occurs in (ii) and (v), with total rate $\sigma_b$; and row (iv) fixes $J$. This gives the second identity.

The total exit rate of $A$ at $(I,J)$ is at most the sum of the exit rates of $Q_p$ at $I$ and $Q_{p+1}$ at $J$, since every joint move changes at least one coordinate and is counted in the corresponding marginal rate. So it is at most two.
\end{proof}

\begin{remark}
In the boundary cases the table degenerates. If $p=0$, then $I=\varnothing$, the lower chain has no moves, and only row (v) can occur. If $|S|=1$, then $S=\{d\}$, the upper chain has no moves, and only row (iv) can occur. In both cases all $E_{ab}$ and $F_{ab}$ vanish, so the mass identities of Theorem~\ref{thm:endpoint}(3), requiring $p>0$ and $q>1$, respectively, are not needed. The coupling is not a coordinatewise domination of exchange rates: in general both $E$ and $F$ are nonzero at the same state.
\end{remark}

\section{The prefix inclusion coupling}\label{sec:inclusion}

\begin{theorem}\label{thm:inclusion}
For every finite poset $P$, every $x\in P$ and every two consecutive feasible prefix sizes $p,p+1$, there is a probability measure $\pi$ on $\mathcal C_p$ whose marginals are $\mu_p$ and $\mu_{p+1}$. That is, the laws of the set of elements preceding $x$, at the two consecutive positions $p+1$ and $p+2$, admit a coupling supported on one-element inclusions.
\end{theorem}

\begin{proof}
The set $\mathcal C_p$ is nonempty: an ideal of $R$ of size $p-|D|$ can be extended by a minimal element of its complement in $R$, which is nonempty since $p+1$ is feasible. By Proposition~\ref{prop:coupled}, $K=\Id+A/2$ is a stochastic matrix on the finite set $\mathcal C_p$. For any probability row vector $\eta$, the Ces\`aro averages $\frac1m\sum_{j=0}^{m-1}\eta K^j$ have a convergent subsequence, and each average differs from its image under $K$ by $(\eta-\eta K^m)/m\to0$. Hence a subsequential limit $\pi$ satisfies $\pi K=\pi$, that is, $\pi A=0$.

Let $\pi_1=\pi\circ\mathrm{pr}_1^{-1}$ be the first marginal. For every $f$ on $\mathcal F_p$, by \eqref{eq:marginal},
\[
\sum_I\pi_1(I)(Q_pf)(I)=\sum_{(I,J)}\pi(I,J)\,A(f\circ\mathrm{pr}_1)(I,J)=\bigl(\pi A\bigr)(f\circ\mathrm{pr}_1)=0.
\]
Taking $f$ to be indicator functions shows $\pi_1Q_p=0$, so $\pi_1=\mu_p$ by Proposition~\ref{prop:exchange}. In the same way the second marginal is $\mu_{p+1}$.
\end{proof}

\begin{remark}
The proof uses neither irreducibility of $A$ nor uniqueness of its invariant law: every invariant law of $A$ projects to $\mu_p$ and $\mu_{p+1}$. All entries of $A$ are rational, so a rational coupling exists, for instance an invariant law of $K$ restricted to one of its closed communicating classes. No width hypothesis was used. Although the generator was constructed for two adjacent layers, its stationary couplings can be glued across all feasible sizes, as follows.
\end{remark}

\begin{corollary}[Simultaneous nested coupling]\label{cor:all-ranks}
Put $p_-=|D|$ and $p_+=|D|+|R|$. There are jointly distributed feasible prefixes
\[
I_{p_-}\subset I_{p_-+1}\subset\cdots\subset I_{p_+}
\]
with $I_p\sim\mu_p$ for every feasible size $p$. Consequently, the conditional expectation of any inclusion-increasing function of the prefix set is nondecreasing in the rank of $x$.
\end{corollary}

\begin{proof}
For each $p_-\le p<p_+$, choose a coupling $\pi^{(p)}$ from Theorem~\ref{thm:inclusion} and extend it by zero off $\mathcal C_p$. Since every $\mu_p(I)$ is positive, the formula
\[
T_p(I,J)=\frac{\pi^{(p)}(I,J)}{\mu_p(I)}
\]
defines a probability transition kernel from $\mathcal F_p$ to $\mathcal F_{p+1}$. Start with $I_{p_-}\sim\mu_{p_-}$ and use these kernels successively. Induction gives
\[
\Prob(I_{p+1}=J)=\sum_{I\in\mathcal F_p}\mu_p(I)T_p(I,J)
=\sum_{I\in\mathcal F_p}\pi^{(p)}(I,J)=\mu_{p+1}(J).
\]
Every positive transition adds one element, proving nesting. An increasing function is therefore nondecreasing along each sampled chain, which proves the expectation assertion. If $R=\varnothing$, there is only one size and the assertion is immediate.
\end{proof}

For example, for any fixed $y\ne x$,
\[
k\longmapsto\Prob\bigl(\pos_L(y)<\pos_L(x)\mid\pos_L(x)=k\bigr)
\]
is nondecreasing on the feasible ranks. The glued coupling concerns prefix sets; it makes no claim that the induced order of the labels added from $R$ is uniform on $\LE(R)$.

\begin{remark}
By Strassen's theorem in its finite form (equivalently, max-flow/min-cut), Theorem~\ref{thm:inclusion} is equivalent to the weighted Hall inequalities $\mu_p(X)\le\mu_{p+1}(N^+(X))$ for every set $X\subseteq\mathcal F_p$, where $N^+(X)$ is the set of $J\in\mathcal F_{p+1}$ containing some member of $X$. This is the form tested in Appendix~\ref{sec:census}.
\end{remark}

\section{Proof of Theorem~\ref{thm:main}}\label{sec:deduction}

Part (1) is Theorem~\ref{thm:inclusion}. We now deduce parts (2)--(4). Positions of $x$ are ranks $k=p+1$, so the coupling of Theorem~\ref{thm:inclusion} couples the prefix laws at ranks $k$ and $k+1$; we write $\pi_k$ for it.

\subsection{Cover observables}

\emph{Lower covers.} A lower cover $v$ of $x$ lies in $D$, hence in every feasible prefix. For a feasible prefix $I$ define
\[
\Phi(I)=\sum_{v\lessdot x}|I|\,\last_{P[I]}(v).
\]
Let $I\subset J=I\cup\{d\}$ be feasible. If $v$ is not maximal in $J$, its term at $J$ vanishes. If $v$ is maximal in $J$, it is maximal in $I$, and Theorem~\ref{thm:endpoint}(3) (with $p=|I|\ge1$) gives $|J|\last_{P[J]}(v)\le|I|\last_{P[I]}(v)$. Hence $\Phi(I)\ge\Phi(J)$.

Conditional on the prefix set $I$, the arrangement of $I$ is uniform by Theorem~\ref{thm:endpoint}(1), so its last element, which is the immediate predecessor of $x$, has law $\last_{P[I]}$. An immediate predecessor of $x$ that is comparable to $x$ is a lower cover. At an empty prefix both $\Phi$ and $L_1$ are zero. Therefore
\[
\E\bigl[\Phi(I)\bigm|\pos_L(x)=k\bigr]=(k-1)\,\frac{L_k}{N_k}.
\]
Integrating $\Phi(I)-\Phi(J)\ge0$ against $\pi_k$ gives, for consecutive ranks with $N_k,N_{k+1}>0$,
\begin{equation}\label{eq:lower}
(k-1)L_kN_{k+1}-kL_{k+1}N_k=N_kN_{k+1}\sum_{(I,J)}\pi_k(I,J)\bigl[\Phi(I)-\Phi(J)\bigr]\ \ge\ 0,
\end{equation}
that is, $(k-1)\ell_k\ge k\ell_{k+1}$.

\emph{Upper covers.} An upper cover $v$ of $x$ lies in every suffix $S_I$. Define
\[
\Psi(S_I)=\sum_{x\lessdot v}|S_I|\,\first_{P[S_I]}(v).
\]
Let $I\subset J=I\cup\{d\}$ be feasible and $S=S_I$, so $S_J=S-d$. If $|S|=1$, then rank $n$ is feasible, so $x$ is maximal and has no upper covers, and $\Psi=0$. Let $|S|\ge2$. If $v$ is not minimal in $S$, its term at $S$ vanishes. If $v$ is minimal in $S$, then $v\ne d$ (as $v\notin R$) and Theorem~\ref{thm:endpoint}(3) gives $(|S|-1)\first_{P[S-d]}(v)\ge|S|\first_{P[S]}(v)$. Hence $\Psi(S_J)\ge\Psi(S_I)$. As before, the first element of the suffix is the immediate successor of $x$, and
\[
\E\bigl[\Psi(S_I)\bigm|\pos_L(x)=k\bigr]=(n-k)\,\frac{U_k}{N_k},
\]
so integrating against $\pi_k$ gives
\begin{equation}\label{eq:upper}
(n-k)u_k\ \le\ (n-k-1)u_{k+1}.
\end{equation}

\subsection{The adjacent-swap identity}
Let $1\le k<n$. A linear extension with $x$ at position $k$ whose immediate successor $y$ is incomparable with $x$ can be mapped to one with $x$ at position $k+1$ by swapping $x$ and $y$; the image has immediate predecessor $y$ incomparable with $x$, and the map is a bijection between these two sets. An immediate neighbour of $x$ that is comparable to $x$ is a cover. Hence, for $1\le k<n$,
\begin{equation}\label{eq:swap}
N_k-U_k=N_{k+1}-L_{k+1}.
\end{equation}
If ranks $k$ and $k+1$ are both feasible, this common value is positive. Indeed, let $J$ be an ideal of $R$ with $|D|+|J|=k-1$, and $y$ a minimal element of $R\setminus J$, which is nonempty since rank $k+1$ is feasible. The predecessors of $y$ lie in $D$ or in $J$ (none is $x$ or above $x$, as $y\inc x$), so $D\cup J\cup\{x,y\}$ is an ideal. A linear extension of $D\cup J$, followed by $x$, $y$ and a linear extension of the rest, is a linear extension with $x$ at rank $k$ followed by the incomparable element $y$. Dividing \eqref{eq:swap} by $N_k$ and $N_{k+1}$ gives
\begin{equation}\label{eq:ratio}
\frac{N_{k+1}}{N_k}=\frac{1-u_k}{1-\ell_{k+1}},
\end{equation}
with positive numerator and denominator. By \eqref{eq:lower}, $\ell_k\ge\frac{k}{k-1}\ell_{k+1}\ge\ell_{k+1}$ for $k\ge2$, and for $k=1$ rank $1$ is feasible only if $x$ is minimal, in which case all $L_j$ vanish. By \eqref{eq:upper}, $u_k\le u_{k+1}$. Hence the ratios \eqref{eq:ratio} are nonincreasing. This proves part~(2).

\subsection{The defect certificate}
For $2\le k<n$, the swap identity at $k-1$ and $k$ gives $L_k-U_{k-1}=N_k-N_{k-1}$ and $U_k-L_{k+1}=N_k-N_{k+1}$. Substitution yields the exact identity
\begin{equation}\label{eq:brackets}
\begin{aligned}
&\bigl[(k-1)L_kN_{k+1}-kL_{k+1}N_k\bigr]+\bigl[kU_kN_k-(k-1)U_{k-1}N_{k+1}\bigr]\\
&\qquad=(k-1)N_{k+1}(L_k-U_{k-1})+kN_k(U_k-L_{k+1})\\
&\qquad=(k-1)N_{k+1}(N_k-N_{k-1})+kN_k(N_k-N_{k+1})=D_k,
\end{aligned}
\end{equation}
When $n\ge2$, the same identity holds at $k=1$: the terms multiplied by $k-1$ vanish, and \eqref{eq:swap} at $k=1$ supplies $U_1-L_2=N_1-N_2$. Accordingly, \eqref{eq:brackets} holds for $1\le k<n$. The cases $k\ge n$, including the singleton poset, are treated directly in the zero-boundary argument below.

Suppose $N_{k-1},N_k,N_{k+1}>0$, and put
\[
r=\frac{N_k}{N_{k+1}},\quad r_0=\frac{N_{k-1}}{N_k},\quad u=u_k,\quad u_0=u_{k-1},\quad \ell=\ell_k,\quad \ell_1=\ell_{k+1},
\]
\[
\Delta u=u-u_0,\qquad \Delta\ell=\ell-\ell_1,\qquad H=(k-1)\ell-k\ell_1,\qquad Q=r-r_0.
\]
By part (2), $H\ge0$, $\Delta\ell\ge0$ and $\Delta u\ge0$ (the latter from \eqref{eq:upper} at $k-1$: $u_0\le\frac{n-k}{n-k+1}u\le u$). By \eqref{eq:ratio} at $k$ and at $k-1$, $r=(1-\ell_1)/(1-u)$ and $r_0=(1-\ell)/(1-u_0)$, with $1-u>0$ and $1-u_0>0$. Writing $1-\ell_1=(1-\ell)+\Delta\ell$ and $1-u_0=(1-u)+\Delta u$,
\[
Q=\frac{(1-\ell_1)(1-u_0)-(1-\ell)(1-u)}{(1-u)(1-u_0)}=\frac{\Delta\ell\,(1-u_0)+\Delta u\,(1-\ell)}{(1-u)(1-u_0)}\ \ge\ 0.
\]
Dividing \eqref{eq:brackets} by $N_kN_{k+1}$, and using $L_k/N_k=\ell$, $L_{k+1}/N_{k+1}=\ell_1$, $U_k/N_{k+1}=ur$ and $U_{k-1}/N_k=u_0r_0$, we get
\[
\frac{D_k}{N_kN_{k+1}}=H+kur-(k-1)u_0r_0.
\]
Since $kur-(k-1)u_0r_0=ur+(k-1)r(u-u_0)+(k-1)u_0(r-r_0)$, this is
\begin{equation}\label{eq:certificate}
\frac{D_k}{N_kN_{k+1}}=H+ur+(k-1)r\,\Delta u+(k-1)u_0\,Q.
\end{equation}
All four terms are nonnegative, and $D_k\ge N_kN_{k+1}\,ur=N_kU_k$.

\subsection{Zero boundaries}
By Theorem~\ref{thm:endpoint}(2), the support of $(N_k)$ is the interval of ranks $|D|+1,\dots,|D|+|R|+1$. If $N_{k+1}=0$, then $D_k=kN_k^2\ge N_kU_k$ because $U_k\le N_k$. If $N_k=0$, then $N_{k-1}N_{k+1}=0$ by interval support, and $D_k=0=N_kU_k$. If $N_{k-1}=0<N_kN_{k+1}$, then $U_{k-1}=0$, so by \eqref{eq:brackets} and \eqref{eq:lower}, $D_k=\bigl[(k-1)L_kN_{k+1}-kL_{k+1}N_k\bigr]+kU_kN_k\ge kN_kU_k\ge N_kU_k$. This covers $k=1$. Together with \eqref{eq:certificate}, this proves part~(3) for every $k\ge1$.

\subsection{Telescoping}
Let $a,h\ge1$. If $N_aN_{a+h}=0$, the right side of \textup{(T)} vanishes: either $N_{a+h}=0$, or $N_a=0<N_{a+h}$, in which case interval support gives $N_{a-1}=0$. Otherwise all $N_j$ with $a\le j\le a+h$ are positive. Dividing property~C at $k$ by $N_kN_{k+1}$ gives $kN_k/N_{k+1}-(k-1)N_{k-1}/N_k\ge1$, and summing over $k=a,\dots,a+h-1$ gives
\[
(a+h-1)\frac{N_{a+h-1}}{N_{a+h}}-(a-1)\frac{N_{a-1}}{N_a}\ \ge\ h.
\]
Multiplying by $N_aN_{a+h}$ gives \textup{(T)}. This proves part~(4), and with it Theorem~\ref{thm:main} and Corollary~\ref{cor:CP104}. \qed

\subsection{The flat equality case of Stanley's inequality}
For completeness, we record explicitly the equality consequence noted in \cite[Proposition~10.6]{ChanPakSurvey}.

\begin{corollary}\label{cor:stanley-equality}
If $N_k>0$, then
\[
N_k^2=N_{k-1}N_{k+1}
\quad\Longleftrightarrow\quad
N_{k-1}=N_k=N_{k+1}.
\]
\end{corollary}

\begin{proof}
Only necessity requires proof. Equality and $N_k>0$ force both neighbouring counts to be positive. Substituting $N_{k-1}N_{k+1}=N_k^2$ into property~C gives $N_k^2\ge N_kN_{k+1}$, hence $N_k\ge N_{k+1}$. Applying property~C to the order dual, at rank $n+1-k$, similarly gives $N_k\ge N_{k-1}$. Since the product of these two neighbours is $N_k^2$, both inequalities must be equalities.
\end{proof}

\appendix
\section{Computational checks}\label{sec:census}
\suppressfloats[t]

Independently of the proof, we checked the conclusions by exact computation. Posets were enumerated up to isomorphism ($1,2,5,16,63,318,2045,16999,183231,2567284$ posets for $n=1,\dots,10$). For each poset and each element $x$, the rank counts were computed exactly from the lattice of order ideals by
\[
N_k=\sum_{\substack{I\text{ ideal},\ |I|=k-1\\ x\text{ minimal in }P-I}}e(I)\,e(P-I-x),
\]
and \textup{(T)} was compared in exact integer arithmetic for all $1\le a,h\le n$. An instance $(x,a,h)$ is \emph{nontrivial} if the right side of \textup{(T)} is positive, and an \emph{equality} if it is nontrivial and both sides agree. The counting code was tested against brute-force enumeration on random posets with $n\le7$.

A self-contained Python script reproducing Tables~\ref{tab:census-all} and~\ref{tab:census-w3} and the coupling check below is included as the ancillary file \texttt{anc/verify\_chan\_pak.py} of the arXiv submission. It generates the posets up to isomorphism by adjoining a maximal element above each order ideal and reducing to a canonical form, checks the number of classes against OEIS A000112, and then carries out the computations described here. A second ancillary file, \texttt{anc/check\_generator.py}, uses the same generator and performs the construction checks described at the end of this appendix.

\begin{table}[htbp]
\centering\small
\begin{tabular}{rrrrrr}
\toprule
$n$ & posets & instances $(x,a,h)$ & nontrivial & equalities & failures\\
\midrule
1 & 1 & 1 & 0 & 0 & 0\\
2 & 2 & 16 & 2 & 2 & 0\\
3 & 5 & 135 & 18 & 15 & 0\\
4 & 16 & 1{,}024 & 135 & 90 & 0\\
5 & 63 & 7{,}875 & 1{,}016 & 523 & 0\\
6 & 318 & 68{,}688 & 8{,}898 & 3{,}443 & 0\\
7 & 2{,}045 & 701{,}435 & 92{,}057 & 26{,}610 & 0\\
8 & 16{,}999 & 8{,}703{,}488 & 1{,}167{,}791 & 252{,}819 & 0\\
9 & 183{,}231 & 133{,}575{,}399 & 18{,}388{,}659 & 3{,}013{,}983 & 0\\
10 & 2{,}567{,}284 & 2{,}567{,}284{,}000 & 363{,}002{,}062 & 45{,}764{,}848 & 0\\
\bottomrule
\end{tabular}
\caption{Inequality \textup{(T)} for every element of every poset with $n\le10$ elements.}
\label{tab:census-all}
\end{table}

\begin{table}[htbp]
\centering\small
\begin{tabular}{rrrrrr}
\toprule
$n$ & posets of width $3$ & instances $(x,a,h)$ & nontrivial & equalities & failures\\
\midrule
11 & 3{,}168{,}869 & 4{,}217{,}764{,}639 & 356{,}132{,}844 & 24{,}384{,}905 & 0\\
12 & 27{,}130{,}482 & 46{,}881{,}472{,}896 & 3{,}758{,}381{,}626 & 202{,}988{,}494 & 0\\
\bottomrule
\end{tabular}
\caption{Inequality \textup{(T)} for every element of every width-three poset with $11$ or $12$ elements.}
\label{tab:census-w3}
\end{table}

We also tested Theorem~\ref{thm:inclusion} directly, without using the construction of Sections~\ref{sec:exchange}--\ref{sec:inclusion}. For every poset with $n\le8$ elements, every element $x$ and every pair of consecutive feasible prefix sizes $p,p+1$, we formed the network with arcs source $\to I$ of capacity $w(I)Z_{p+1}$, $I\to J$ for $I\subset J$ of infinite capacity, and $J\to$ sink of capacity $w(J)Z_p$, where $Z_p=\sum_{|I|=p}w(I)$. A coupling supported on inclusions exists if and only if the maximum flow equals $Z_pZ_{p+1}$. This held in all cases ($19{,}449$ posets).

Finally, we checked the construction of Sections~\ref{sec:endpoint}--\ref{sec:inclusion} itself, in exact rational arithmetic, for every poset with $n\le8$ elements, every element $x$ and every pair of consecutive feasible prefix sizes. For each nested pair $(I,J)$ we verified the endpoint inequalities and mass identities of Theorem~\ref{thm:endpoint}(3); for each layer, the legality of positive-rate exchanges, the exit-rate bound, detailed balance with respect to $\mu_p$, and connectivity (Proposition~\ref{prop:exchange}). We then built the generator $A$ from the table of Section~\ref{sec:coupled} and verified that rows (iv) and (v) are nonnegative (Lemma~\ref{lem:budget}), that every row with positive rate leads to a state of $\mathcal C_p$, and that aggregating the joint rates by destination reproduces every rate of $Q_p$ and $Q_{p+1}$ exactly, which is \eqref{eq:marginal}. Finally, on every closed communicating class of $A$ we solved $\pi A=0$ exactly and compared both marginals with $\mu_p$ and $\mu_{p+1}$. At $n=8$ this covers $456{,}018$ pairs of consecutive layers and $1{,}735{,}666$ nested states, with $6{,}981{,}118$ positive joint rates. We ran the same checks on a uniformly random sample of $20{,}000$ of the $183{,}231$ posets with $9$ elements (fixed seed), covering $706{,}230$ pairs of consecutive layers and $3{,}311{,}696$ nested states. No check failed. In every case $A$ had a single closed class. The script \texttt{anc/check\_generator.py} performs these checks.

\FloatBarrier
\section*{Acknowledgements}
This work was carried out with substantial assistance from AI systems. The principal one was EmiOS, the author's research harness, which coordinates large language models to develop, record and verify mathematical arguments; much of the underlying work was produced and checked within it. OpenAI GPT models, used both inside the harness and directly, and Anthropic Claude were used to search for and develop the arguments, to check proofs, to write and run the verification code of Appendix~\ref{sec:census}, and to prepare the manuscript. The author takes responsibility for the mathematical claims and the final manuscript.

\end{document}